\documentclass[oneside,a4paper]{article}
\usepackage{amssymb,amsmath,latexsym,graphicx,tikz,indentfirst,scalerel,hyperref,caption,marginnote,multicol}

\hypersetup{
  colorlinks   = true, 
  urlcolor     = blue, 
  linkcolor    = blue, 
  citecolor   = red 
}

\usetikzlibrary{automata, positioning, calc,arrows}

\def\dj{d\kern-0.4em\char"16\kern-0.1em}
\def\Dj{\mbox{\raise0.3ex\hbox{-}\kern-0.4em D}}

\newtheorem{theorem}{Theorem}
\newtheorem{lemma}[theorem]{Lemma}

\newtheorem{corollary}[theorem]{Corollary}
\newtheorem{definition}[theorem]{Definition}

\makeatletter
\renewcommand{\@dotsep}{10000}
\makeatother

\newenvironment{proof}
{\noindent
{\it Proof.}}
{\hspace{\stretch{1}}%
$\Box$}

\newcounter{primer}[section]

\DeclareMathOperator{\Is}{Is}

\DeclareMathOperator{\Cen}{Cen}

\makeatletter 
\tikzset{my loop/.style =  {to path={
  \pgfextra{}
  [looseness=4,min distance=5mm]
  \tikz@to@curve@path},font=\sffamily\small
  }}  
\makeatletter 

\newcommand{\rightarrowp}{\stackrel{+}{\rightarrow}}
\newcommand{\rightarrowpm}{\stackrel{\pm}{\rightarrow}}
\newcommand{\simp}{\stackrel{p}{\sim}}
\newcommand{\simpp}{\stackrel{p_+}{\sim}}
\newcommand{\simppm}{\stackrel{p_\pm}{\sim}}

\makeatletter
\newcommand{\myitem}[1]{%
\item[#1]\protected@edef\@currentlabel{#1}%
}
\makeatother

\begin{document}

\thispagestyle{empty}
\begin{center}
\Large{The Power Graph of a Torsion-Free Group of Nilpotency Class $2$}
\vspace{5mm}

\small{Samir Zahirovi\' c\\
{\it Department of Mathematics and Informatics, Faculty of Sciences,\\
University of Novi Sad, Serbia\\
\href{mailto:samir.zahirovic@dmi.uns.ac.rs}{samir.zahirovic@dmi.uns.ac.rs}}}
\end{center}

\begin{abstract}
The directed power graph $\mathcal G(\mathbf G)$ of a group $\mathbf G$ is the simple digraph with vertex set $G$ in which $x\rightarrow y$ if $y$ is a power of $x$, the power graph is the underlying simple graph, and the enhanced power graph of $\mathbf G$ is the simple graph with the same vertex set such that two vertices are adjacent if they are powers of some element of $\mathbf G$.

In this paper three versions of the definition of the power graphs are discussed, and it is proved that the power graph by any of the three versions of the definitions determines the other two up to isomorphism. It is also proved that, if $\mathbf G$ is a torsion-free group of nilpotency class $2$ and if $\mathbf H$ is a group such that $\mathcal G(\mathbf H)\cong\mathcal G(\mathbf G)$, then $\mathbf G$ and $\mathbf H$ have isomorphic directed power graphs, which was an open problem proposed by Cameron, Guerra and Jurina \cite{cameron3}.
%
%
\end{abstract}

\section{Introduction}


The directed power graph of a group is the simple digraph with its the vertex set being the universe of the group and in which $x\rightarrow y$ if $y\in\langle x\rangle$; the power graph of a group is the underlying simple graph.  The directed power graph was introduced by Kelarev and Quinn \cite{prvi kelarev i kvin}, while the power graph was first studied by Chakrabarty, Ghosh and Sen \cite{cakrabarti}. By results of \cite{power graph 2}, the power graph of a finite group determines the directed power graph up to isomorphism. The power graph has been studied by many authors, including \cite{power graph, power graph 2, cameron3, on the structure, sitov, drugi kelarev i kvin, kelarev3, kelarev4, cameron4}. The reader is referred to the survey \cite{pregled za power grafove} for more details. 

Beside the just mentioned definition for the power graphs, many authors, including \cite{prvi kelarev i kvin}, often define the power graph as the graph in which two vertices are adjacent if there exists a positive integer $n$ such that $y=x^n$ or $x=y^n$; we shall call this graph the $N$-power graph of the group. Beside this definition, authors of \cite{cameron3} brought the definition of the power graph by which two vertices are adjacent if there is a non-zero integer such that $y=x^n$ or $x=y^n$, and the graph defined in this manner we will call the $Z^\pm$-power graph. In the case of torsion groups it is obvious that all three of these definitions of the power graph produce the same graph. It is also easily seen that, in the case of torsion-free groups, the power graph and the $Z^\pm$-power graph determine each other. In Section \ref{sekcija o raznim definicijama} it is proved that the power graph, the $N$-power graph and the $Z^\pm$-power graph of any group determine each other.

In Section \ref{torziono slobodne grupe} the power graph of a torsion-free group is investigated. Cameron, Guerra and Jurina \cite{cameron3} showed that, if in a torsion-free group $\mathbf G$ every element is contained in a unique maximal cyclic subgroup, and if $\mathcal G(\mathbf G)\cong\mathcal G(\mathbf H)$, then $\mathbf G$ and $\mathbf H$ have isomorphic directed power graphs. They also showed that, for any pair of torsion-free groups of nilpotency class $2$, isomorphism of the power graphs implies the isomorphism of the directed power graphs. The authors of \cite{cameron3} asked whether this implication also holds when at least one of the groups is torsion-free and of nilpotency class $2$, and Section \ref{torziono slobodne grupe} answers this question affirmatively. Beside that, in \cite{cameron3} it was proved that there is not group non-isomorphic to $\mathbb Z$ whose power graph is isomorphic to $\mathcal G(\mathbb Z)$. In Section \ref{torziono slobodne grupe} the same is proved for the group of rationals.




\section{Basic notions and notations}

{\bf Graph} $\Gamma$ is a structure $(V(\Gamma),E(\Gamma))$, or simply $(V,E)$, where $V$ is a set, and $E\subseteq V^{[2]}$ is a set of two-element subsets of $V$. Set $V$ is called the set of vertices, while $E$ is called the set of edges. We say that vertices $x$ and $y$  are adjacent in $\Gamma$ if $\{x,y\}\in E$, and we denote it with $x\sim_\Gamma y$, or simply $x\sim y$. Graph $\Delta=(V_1,E_1)$ is said to be a subgraph of {\bf subgraph} of graph $\Gamma=(V_2,E_2)$ if $V_1\subseteq V_2$ and $E_1\subseteq E_2$. $\Delta$ is an {\bf induced subgraph} of $\Gamma$ if $V_1\subseteq V_2$, and $E_1=E_2\cap V_1^2$. In this case we also say that graph $\Delta$ is an induced subgraph of $\Gamma$ by $V_1$, and we write this fact by $\Delta=\Gamma[V_1]$. The {\bf strong product} of graphs $\Gamma$ and $\Delta$ is the graph $\Gamma\boxtimes\Delta$ such that
\begin{align*}
(x_1,y_1)\sim_{\Gamma\boxtimes\Delta}(x_2,y_2)\text{ if }&(x_1=x_2\wedge y_1\sim_\Delta y_2)\\
&\vee(x_1\sim_\Gamma x_2\wedge y_1= y_2)\\
&\vee(x_1\sim_\Gamma x_2\wedge y_1\sim_\Delta y_2).
\end{align*}


{\bf Directed graph} (or {\bf digraph}) $\vec \Gamma$ is a structure $(V(\vec\Gamma),E(\vec\Gamma))$, or simply $(V,E)$, where $V$ is a set, and $E$ is an irreflexive relation on $V$. Here $V$ and $E$ are the set of vertices and the set of edges, respectively. For vertices $x$ and $y$ of $\Gamma$ such that $(x,y)\in E$, we say that there is directed edge from $x$ to $y$, and we denote it with with $x\rightarrow_{\vec\Gamma} y$, or simply $x\rightarrow y$. In this case, we also say that $y$ is a direct successor of $x$, and that $x$ is a direct predecessor of $y$.

Throughout this paper, we shall denote algebraic structures, such as groups or loops, with bold capitals, while we will denote their universes with respective regular capital letters.  For $x,y\in G$, where $\mathbf G$ is a group, we shall write $x\approx_{\mathbf G} y$, or simply $x\approx y$, if $\langle x\rangle=\langle y\rangle$. Also, $o(x)$ shall denote the order of an element $x$ of a group. Throughout this paper, will will say that a loop is {\bf power-associative} if every its subloop generated by one element is a group.

\begin{definition}\label{definicije pridruzenih grafova}
The {\bf directed power graph} of a group $\mathbf G$ is the digraph $\vec{\mathcal G}(\mathbf G)$ whose vertex set is $G$, and in which there is a directed edge from $x$ to $y$, $x\neq y$, if there exists $n\in\mathbb Z$ such that $y=x^n$. If there is a directed edge from $x$ to $y$ in $\vec{\mathcal G}(\mathbf G)$, we shall denote it with $x\rightarrow_{\mathbf G} y$, or shortly with $x\rightarrow y$.

The {\bf power graph} of a group $\mathbf G$ is the graph $\mathcal G(\mathbf G)$ whose vertex set is $G$, and whose vertices $x$ and $y$, $x\neq y$, are adjacent if there exists $n\in\mathbb Z$ such that $y=x^n$ or $x=y^n$. If $x$ and $y$ are adjacent in $\mathcal G(\mathbf G)$, we shall denote it with $x\simp_{\mathbf G} y$, or shortly with $x\simp y$.
\end{definition}

For a graph $\Gamma$ and its vertex $x$, $\overline N_{\Gamma}(x)$ denotes the closed neighborhood of $x$ in the graph $\Gamma$. If $\overline N_\Gamma(x)=\overline N_\Gamma(y)$, for vertices $x$ and $y$ of $\Gamma$, we shall shortly write as $x\equiv_{\Gamma} y$, and, if $\Gamma$ is the power graph of a group $\mathbf G$, then we shall write $x\equiv_{\mathbf G} y$ instead of $x\equiv_{\mathcal G(\mathbf G)} y$. 


Notice that, in the power graph of a group, the identity element is adjacent to all other vertices of the graph. Therefore, eccentricity of each of these graphs is $1$. Because of that, in the case of the power graph and the enhanced power graph, it is justifiable to call the set of vertices, that are adjacent to all vertices of the graph other than itself, its center. For a group $\mathbf G$, the set
$$\{x\in G\mid x\sim_{\Gamma} y\text{ for all }y\in G\setminus\{x\}\}$$
we will call the {\bf center of the power graph} $\mathcal G(\mathbf G)$, denoted by $\Cen(\mathcal G(\mathbf G))$.

Let us also note that the definitions of the directed power graph and the power graph can be applied not only on groups, but on power-associative loops too in the same manner. In this case, the center of the power graph and the center of the enhanced power graph are defined in the analogous way.

It is easily seen that the directed power graph of a group determines the power graph, and Cameron \cite{power graph 2} proved that, in the case of finite groups, the power graph determines the directed power graph too. Cameron, Guerra and Jurina \cite{cameron3} proved the same for some classes of torsion-free groups as well.

\section{On different definitions of the power graph}
\label{sekcija o raznim definicijama}

This section deals with three different versions of the definition of the power graph. Namely, by Definition \ref{definicije pridruzenih grafova}, which is consistent with the definitions from \cite{on the structure} and \cite{cameron3}, for elements $x$ and $y$ of a group $\mathbf G$ holds $x\rightarrow y$ in $\vec{\mathcal G}(\mathbf G)$ if there exists $n\in \mathbb Z$ such that $y=x^n$, and $x$ and $y$ are adjacent in $\mathcal G(\mathbf G)$ if there exists $n\in\mathbb Z$ such that $y=x^n$ and $x=y^n$. Let us introduce the definition of the power graph that was introduced in \cite{prvi kelarev i kvin}. To avoid any confusion, we bring the terms of $N$-power graph and directed $N$-power graph. 

\begin{definition}
The {\bf directed $N$-power graph} of a group $\mathbf G$ is the digraph $\vec{\mathcal G}^+(\mathbf G)$ whose vertex set is $G$, and in which there is a directed edge from $x$ to $y$, $x\neq y$, if there exists $n\in\mathbb N$ such that $y=x^n$. If there is a directed edge from $x$ to $y$ in $\vec{\mathcal G}^+(\mathbf G)$, we shall denote it with $x\rightarrowp_{\mathbf G} y$, or shortly with $x\rightarrowp y$.

The {\bf $N$-power graph} of a group $\mathbf G$ is the graph $\mathcal G^+(\mathbf G)$ whose vertex set is $G$, and whose vertices $x$ and $y$, $x\neq y$, are adjacent if there exists $n\in\mathbb N$ such that $y=x^n$ or $x=y^n$. If $x$ and $y$ are adjacent in $\mathcal G^+(\mathbf G)$, we shall denote it with $x\simpp_{\mathbf G} y$, or shortly with $x\simpp y$.
\end{definition}

An advantage of the above definition is that it can be applied to any power-associative groupoid where inverse elements, or even the identity element, might not exist.  On the other hand, in the case of torsion groups, $N$-power graph is the same as the power graph.

Beside that, in \cite{cameron3} authors used another version of the definition in which they insisted on the exponent from the expression $y=x^n$ to be non-zero integer. The graph defined in this manner we shall call $Z^\pm$-power graph.

\begin{definition}
The {\bf directed $Z^\pm$-power graph} of a group $\mathbf G$ is the digraph $\vec{\mathcal G}^{\pm}(\mathbf G)$ whose vertex set is $G$, and in which there is a directed edge from $x$ to $y$, $x\neq y$, if there exists $n\in\mathbb Z\setminus\{0\}$ such that $y=x^n$. If there is a directed edge from $x$ to $y$ in $\vec{\mathcal G}^\pm(\mathbf G)$, we shall denote it with $x\rightarrowpm_{\mathbf G} y$, or shortly with $x\rightarrowpm y$.

The {\bf $Z^\pm$-power graph} of a group $\mathbf G$ is the graph $\mathcal G^{\pm}(\mathbf G)$ whose vertex set is $G$, and in which $x$ and $y$, $x\neq y$, are adjacent if there exists $n\in\mathbb Z\setminus\{0\}$ such that $y=x^n$ or $x=y^n$. If $x$ and $y$ are adjacent in $\mathcal G^\pm(\mathbf G)$, we shall denote it with $x\simppm_{\mathbf G} y$, or shortly with $x\simppm y$.
\end{definition}

Let us note that the directed $Z^\pm$-power graph and the $Z^\pm$-power graph of power-associative loops  are defined in the same way. This version of the definition of the power graphs helped authors of \cite{cameron3} to make their arguments simpler while studying the power graph of torsion-free groups, while, in the case of torsion-free groups, the power graph and the $Z^\pm$-power graph determine each other. Beside that, it's easily noticed that the directed power graph and the directed $Z^\pm$-power graph of a group determine each other. In this section it is proved that the power graph, the $N$-power graph, and the $Z^\pm$-power graph of any power-associative loop determine each other.

\begin{lemma}\label{pomocna za nenula i obicni stepeni graf}
Let $\mathbf G$ be a power-associative loop such that $\lvert\Cen(\mathcal G(\mathbf G))\rvert>1$. Then the following hold:
\begin{enumerate}
\item $\mathbf G$ is the infinite cyclic group, or all elements of the loop $\mathbf G$ have finite orders.
\item $\mathbf G\cong(\mathbb Z,+)$ if and only if $G$ is union of countably many $\equiv_{\mathbf G}$-classes of cardinality $2$ and one $\equiv_{\mathbf G}$-class of cardinality $3$.
\end{enumerate}
\end{lemma}

\begin{proof}
Let $\mathbf G$ be a power-associative loop, and let $\lvert\Cen(\mathcal G(\mathbf G))\rvert>1$. Notice that in the power graph no non-identity element of finite order is adjacent to an element of infinite order. Therefore, either all non-identity elements of $\mathbf G$ have finite order, or all elements of $\mathbf G$ are of infinite order. Suppose that $\mathbf G$ doesn't contain any non-identity element of finite order. Since  $\lvert\Cen(\mathcal G(\mathbf G))\rvert>1$, then there exists $x\in G\setminus\{e_{\mathbf G}\}$ which is adjacent in $\mathcal G(\mathbf G)$ with all elements of $G\setminus\{x\}$. Let us show that then $\mathbf G=\langle x\rangle$. If $\mathbf G\neq\langle x\rangle$, then there exist $y\in G\setminus\langle x\rangle$ and $n\in\mathbb Z\setminus\{-1,0,1\}$ such that $x=y^n$. Then for $m\in\mathbb N$ relatively prime to $n$ holds $x\not\simp y^m$, which is a contradiction. Therefore, $\mathbf G$ contains only elements of finite order or  $\mathbf G\cong(\mathbb Z,+)$, and thus, (a) has been proved.

If $\mathbf G\cong(\mathbb Z,+)$ is generated by $x$, then $x$, $x^{-1}$ i $e_{\mathbf G}$ make up an $\equiv_{\mathbf G}$-class of cardinality $3$. Let $y\in G\setminus\{x,x^{-1},e_{\mathbf G}\}$. Obviously, $y\equiv_{\mathbf G} y^{-1}$. Suppose that there exists $z\not\in\{ y,y^{-1}\}$ such that $z\equiv_{\mathbf G} y$. Then there exists $n\in\mathbb Z\setminus\{-1,0,1\}$ such that $z=y^n$ or $y=z^n$. Then, for $m$ relatively prime to $n$, follows $z\not\simp y^m\simp y$ or $y\not\simp z^m\simp z$, which is a contradiction. Therefore, any non-identity element of $\mathbf G$ which isn't a generator of $\mathbf G$ is contained in an $\equiv_{\mathbf G}$-class of cardinality $2$. This proves one implication of (b). In order to prove the other one as well, suppose that $\mathbf G\not\cong(\mathbb Z,+)$, which, by (a), implies that $\mathbf G$ contains only elements of finite order. Were there an element $x$ of $\mathbf G$ of order $n>6$, then  $\langle x\rangle$ would have more than $2$ generators, and these generators of $\mathbf G$ would have the same closed neighborhoods in $\mathcal G(\mathbf G)$. Therefore, $\mathbf G$ contains only elements of order at most $6$, and it contains countably many element of order $3$ or $6$.  It follows that $\mathcal G(\mathbf G)$ contains an $\equiv_{\mathbf G}$-class of cardinality $1$ which contains the indentity element of $\mathbf G$. This proves (b), and thus the lemma has been proved.
\end{proof}

\begin{theorem}\label{medjusobno odredjivanje opsteg i nenula-stepenog grafa}
Let $\mathbf G$ and $\mathbf H$ be power-associative loops. Then $\mathcal G(\mathbf G)\cong\mathcal G(\mathbf H)$ if and only if $\mathcal G^{\pm}(\mathbf G)\cong\mathcal G^{\pm}(\mathbf H)$. 
\end{theorem}

\begin{proof}
Let $\mathbf G$ be power-associative loop, and let $\Gamma=\mathcal G(\mathbf G)$, $\Gamma^{\pm}=\mathcal G^{\pm}(\mathbf G)$,  $\Delta=\mathcal G(\mathbf H)$, and $\Delta^{\pm}=\mathcal G^{\pm}(\mathbf H)$. Trivially, $\Gamma^{\pm}\subseteq\Gamma$ and $E(\Gamma)\setminus E(\Gamma^{\pm})=\big\{\{e,x\}\mid x\in G\text{ and }o(x)=\infty\big\}$, and similar holds for $\Delta$ and $\Delta^{\pm}$. 

If $\lvert\Cen(\Gamma)\rvert>1$, then, by Lemma \ref{pomocna za nenula i obicni stepeni graf}, if $\mathbf G\cong(\mathbb Z,+)$, it follows $\mathbf H\cong(\mathbb Z,+)$, which trivially implies the stated equivalence. Similarly, if $\mathbf G\not\cong(\mathbf Z,+)$, then $\mathbf G$ and $\mathbf H$ don't have any element of infinite order. Then $\Gamma=\Gamma^{\pm}$ and $\Delta=\Delta^{\pm}$, so the equivalence holds in this case too.

Suppose now that $\lvert\Cen(\Gamma)\rvert=1$. Let $\Gamma\cong\Delta$, and let $\varphi:G\rightarrow H$ be an isomorphism from $\Gamma$ to $\Delta$. Let us show that $\varphi$ is an isomorphism from $\Gamma^{\pm}$ to $\Delta^{\pm}$ too. Let $x,y\in G$, and let $x\simppm y$. Obviously, then $\varphi(x)\simp \varphi(y)$. Suppose that $\varphi(x)\not\simppm\varphi(y)$. Then, without loss of generality, the order of $\varphi(x)$ is infinite and $\varphi(y)=e_{\mathbf H}$. Then $\varphi(x)$ is contained in a connected component of $\Delta\setminus\{e_{\mathbf H}\}$ whose vertex set is union of countably many $\equiv_{\Delta}$-classes of cardinality two and which contains an infinite clique. Obviously, similar holds for $x$ too. If the order of $x$ was finite, then connected component of $\Gamma\setminus\{e_{\mathbf G}\}$ containing $x$ has no element of order greater than $6$, so it doesn't contain any maximal clique. Therefore, the order of $x$ is infinite. Obviously, $y=e_{\mathbf G}$ because $\varphi(y)=e_{\mathbf H}\in\Cen\big(\mathcal G(\mathbf H)\big)$. Therefore, $x\not\simppm y$, which is a contradiction. This proves that $\varphi(x)\simppm\varphi(y)$. Similarly, $\varphi(x)\simppm\varphi(y)$ implies $x\simppm y$.

It remains to be shown that, in the case when $\lvert\Cen(\Gamma)\rvert=1$, iso\-mor\-phism of the $Z^\pm$-power graphs implies iso\-mor\-phism of the power graphs of $\mathbf G$ and $\mathbf H$. Let $\Gamma^{\pm}\cong\Delta^{\pm}$, and $\varphi:G\rightarrow H$ isomorphism from $\Gamma^{\pm}$ to $\Delta^{\pm}$. It is possible that $\varphi(e_{\mathbf G})\neq e_{\mathbf H}$, so let $\hat\varphi(x)=\tau(\varphi(x))$, where $\tau:H\rightarrow H$ is the transposition of $e_{\mathbf H}$ and $\varphi(e_{\mathbf G})$ if $\varphi(e_{\mathbf G})\neq e_{\mathbf H}$, and the identity mapping otherwise. Because a connected component $\Phi$ of $\Delta^{\pm}$ contains elements of infinite order if and only if $V(\Phi)$ is union of countably many $\equiv_{\Delta^{\pm}}$-classes of cardinality $2$ and $\Phi$ contains an infinite clique, then $\varphi(e_{\mathbf G})$ is contained in the connected component that contains $e_{\mathbf H}$, so $\hat\varphi$ is an isomorphism from $\Gamma^{\pm}$ to $\Delta^{\pm}$. Let us show that $\hat\varphi$ is isomorphism from $\Gamma$ to $\Delta$. Let $x,y\in G$, and let $x\simp y$. If $x\simppm y$, then obviously follows $\hat\varphi(x)\simp \hat\varphi(y)$. Suppose that $x\not\simppm y$. Then, without loss of generality, $x=e_{\mathbf G}\in\Cen\big(\mathcal G(\mathbf G)\big)$, and because $\hat\varphi(x)=e_{\mathbf H}$, it follows that $\hat\varphi(x)\simp\hat\varphi(y)$. In the same manner is proved that $\hat\varphi(x)\simp\hat\varphi(y)$ implies $x\simp y$. Therefore, $\hat\varphi$ is an isomorphism from $\Gamma$ to $\Delta$. Thus, the theorem has been proved.
\end{proof}\\

With the previous theorem we have proved that the power graph and the $Z^\pm$-power graph of a power-associative loop carry the same amount of information about the original structure. Let us show that the same holds for the $Z^\pm$-power graph and the $N$-power graph.

\begin{lemma}\label{prva lema za nenula i pozitivno-stepeni graf}
Let $\mathbf G$ be a power-associative loop. Then, for each element $x\in G$ of infinite order, $x$ and $x^{-1}$ lie in different connected components of $\mathcal G^+(\mathbf G)$.
\end{lemma}

\begin{proof}
Let $x\in G$ be an element of infinite order. For natural numbers $n$ and $m$, let $S(x,n,m)=\{y\mid x^n=y^m\}$, and let $\overline S(x)=\bigcup_{n,m\in\mathbb N}S(x,n,m)$. Let us show that the set $\overline S(x)$ induces a connected component of $\mathcal G^+(\mathbf G)$. Let $y\in\overline S(x)$. Then $y\in S(x,n,m)$ for some $n,m\in\mathbb N$, and suppose that $z\simpp y$ for some $z\in G$. If $z\rightarrowp y$, it is easily noticed that $z\in\overline S(x)$, so suppose that $y\rightarrowp z$. Then $z=y^k$ for some $k\in\mathbb N$. This implies that $z\in S(x,nk,mk)\subseteq \overline S(x)$. Therefore, because $\overline S(x)$ induces a connected subgraph of $\mathcal G^+(\mathbf G)$, it follows that $\overline S(x)$ induces a connected component of $\mathcal G^+(\mathbf G)$, while clearly $x^{-1}\not\in \overline S(x)$. This proves the lemma.
\end{proof}\\

We remind the reader that, for graphs $\Gamma$ and $\Delta$, $\Gamma\boxtimes\Delta$ denotes the strong graph product of $\Gamma$ and $\Delta$, and, for $X\subseteq V(\Gamma)$, $\Gamma[X]$ denotes the subgraph of $\Gamma$ induced by $X$. In this paper $P_2$ shall denote the path with two vertices.

\begin{lemma}\label{druga lema za nenula i pozitivno-stepeni graf}
Let $\mathbf G$ be a power-associative loop. For each connected component $\Phi$ of $\Gamma^\pm=\mathcal G^\pm(\mathbf G)$ containing elements of infinite order, there exist connected components $\Psi_1$ and $\Psi_2$ of $\mathcal G^+(\mathbf G)$ such that:
\begin{multicols}{2}
\begin{enumerate}
\item $V(\Phi)=V(\Psi_1)\cup V(\Psi_2)$;
\item $\Psi_1\cong \Psi_2$; 
\item $\Phi\cong \Psi_1\boxtimes P_2$; 
\item $\Psi_1\cong\Phi/\mathord{\equiv_{\Gamma^\pm}}$.
\end{enumerate}
\end{multicols}
\end{lemma}

\begin{proof}
Let us denote graphs $\mathcal G^+(\mathbf G)$ and $\mathcal G^\pm(\mathbf G)$ with  $\Gamma^+$ i $\Gamma^\pm$, respectively. Let $x\in G$ be an element of infinite order, and let $T(x,n,m)=\{y\mid y^m=x^n\}$, for any $x\in G$, $n\in\mathbb Z\setminus\{0\}$, and $m\in\mathbb N$. Let $\overline T(x)=\bigcup_{n\in\mathbb Z\setminus\{0\},m\in\mathbb N}T(x,n,m)$. It is easily seen that $\overline T(x)$ induces a connected subgraph of $\mathcal G^\pm(\mathbf G)$. Also, similarly like in the proof of Lemma \ref{prva lema za nenula i pozitivno-stepeni graf}, if $z\simppm y$ for some $y\in\overline T(x)$, then $z\in\overline T(x)$. Therefore, $\overline T(x)$ induces a connected component of $\Gamma^\pm$. Now, it is easily seen that $\overline T(x)=\overline S(x)\cup\overline S(x^{-1})$, where $\overline S(x)$ is the set defined in the proof of Lemma \ref{prva lema za nenula i pozitivno-stepeni graf}. As seen in the proof of Lemma \ref{prva lema za nenula i pozitivno-stepeni graf}, $\overline S(x)$ and $\overline S(x^{-1})$ induce different connected components of $\Gamma^+$. Thus, each connected component of $\Gamma^\pm$ containing elements of infinite order is union of two connected components $\Psi_1$ and $\Psi_2$ of $\Gamma^+$, where $\Psi_2$ contains inverses of $\Psi_1$. This proves {\it 1}.

Let $\Phi$ be the connected component of $\Gamma^\pm$ induced by $V(\Psi_1)\cup V(\Phi_2)$. Since the mapping $x\mapsto x^{-1}$ is an automorphism of $\Gamma^+$ which maps $V(\Psi_1)$ onto $V(\Psi_2)$, then $\Psi_1\cong \Psi_2$. This proves {\it 2}. Further, because $x$ and $x^{-1}$ have the same closed neighborhood in $\Gamma^\pm$, and because $\Psi_1$ and $\Psi_2$ are induced subgraphs of $\Phi$, then $\Phi\cong \Psi_1\boxtimes P_2$. Beside that, $V(\Phi)$ is union of $\equiv_{\Gamma^\pm}$-classes of cardinality $2$, each of which contains an element of infinite order and its inverse. Therefore, $\Psi_1$ has one vertex from each $\equiv_{\Gamma^\pm}$-class contained in $V(\Phi)$, so $\Psi_1$ is isomorphic to the graph constructed by replacing each $\equiv_{\Gamma^\pm}$-class contained in $\Phi$ with one vertex. Thus, {\it 3.} and {\it 4.} have been proved as well, which finishes our proof.
\end{proof}

\begin{theorem}\label{medjusobno odredjedjivanje pozitivno-stepenog i nenula-stepenog grafa}
Let $\mathbf G$ and $\mathbf H$ be power-associative loops. Then $\mathcal G^+(\mathbf G)\cong\mathcal G^+(\mathbf H)$ if and only if $\mathcal G^\pm(\mathbf G)\cong\mathcal G^\pm(\mathbf H)$.
\end{theorem}

\begin{proof}
Let us denote graphs $\mathcal G^+(\mathbf G)$, $\mathcal G^\pm(\mathbf G)$, $\mathcal G^+(\mathbf H)$ and $\mathcal G^\pm(\mathbf H)$ with  $\Gamma^+$, $\Gamma^\pm$, $\Delta^+$ and $\Delta^\pm$, respectively. Let $G^{<\infty}$ and $H^{<\infty}$ be be the sets of all elements of finite order of loops $\mathbf G$ and $\mathbf H$.

Suppose that  $\mathcal G^+(\mathbf G)\cong\mathcal G^+(\mathbf H)$. Set $G^{<\infty}$ induced connected component $\Gamma^+$ and $\Gamma^\pm$. What is more, $G^{<\infty}$ induces the only connected component of $\Gamma^+$ which isn't union of  $\equiv_{\Gamma^+}$-classes of cardinality $1$ or which doesn't contain any infinite clique. Since similar holds for the connected component of $\Delta^+$ induced by $H^{<\infty}$, then $\Gamma^\pm[G^{<\infty}]=\Gamma^+[G^{<\infty}]\cong \Delta^+[H^{<\infty}]=\Delta^\pm[H^{<\infty}]$. Now, let $\Phi$ be a connected component $\Gamma^\pm$ containing elements of infinite order, and let $\kappa$ be the cardinality of the set of all connected components of $\Gamma^\pm$ isomorphic to $\Phi$. By Lemma \ref{druga lema za nenula i pozitivno-stepeni graf}, $\Phi$ is induced by vertices of two isomorphic connected components $\Psi_1$ and $\Psi_2$ of $\Gamma^+$, so the cardinality of the set of all connected components of $\Gamma^+$ isomorphic to $\Psi_1$ is $2\kappa$. Since $\Gamma^+\cong\Delta^+$, then $\Delta^+$ exactly $2\kappa$ connected components isomorphic to $\Psi_1$, so, by Lemma \ref{druga lema za nenula i pozitivno-stepeni graf}, graph $\Delta^\pm$ contains $\kappa$ connected components isomorphic to $\Phi$. Similar holds for each connected component of $\Delta^\pm$. Therefore, since $\Gamma^\pm$ and $\Delta^\pm$ have the same, up to isomorphism, connected components, and, for each connected component $\Phi$ of $\Gamma^\pm$, the cardinality of the set of all connected components components of $\Gamma^\pm$ isomorphic to $\Phi$ is equal to the cardinality of the set of all connected components of $\Delta^\pm$ isomorphic to $\Phi$, it follows that $\Gamma^\pm$ and $\Delta^\pm$ are isomorphic.

The other implication is proved in a similar manner. Let $\mathcal G^\pm(\mathbf G)\cong\mathcal G^\pm(\mathbf H)$. Now $G^{<\infty}$ is the only connected component of $\Gamma^\pm$ which isn't union of $\equiv_{\Gamma^\pm}$-classes of cardinality $2$ or doesn't contain any infinite clique, and similar holds for the connected component of $\Delta^\pm$ induced by $H^{<\infty}$. Therefore, $G^{<\infty}$ and $H^{<\infty}$ induce isomorphic connected components of $\Gamma^+$ and $\Delta^+$, respectively. Now, let $\Psi$ be a connected component of $\Gamma^+$ containing elements of infinite order. By Lemma \ref{prva lema za nenula i pozitivno-stepeni graf}, there exists a cardinal $\kappa$ such that the cardinality of the set of all connected components of $\Gamma^+$ isomorphic to $\Psi$ is equal to $2\kappa$. By Lemma \ref{druga lema za nenula i pozitivno-stepeni graf}, $\Gamma^\pm$ contains $\kappa$ connected components isomorphic to $\Psi\boxtimes P_2$, where $P_2$ is the path with $2$ vertices, so $\Delta^\pm$ contains $\kappa$ connected components isomorphic to $\Psi\boxtimes P_2$ too. Therefore, by Lemma \ref{druga lema za nenula i pozitivno-stepeni graf}, graph $\Delta^+$ contains $2\kappa$ connected components isomorphic to $\Psi$. Thus, $\Delta^+$ and $\Gamma^+$ contain, up to isomorphism, same connected components, and for each connected component $\Psi$ of $\Gamma^+$, the cardinality of the set of all connected components of $\Gamma^+$ isomorphic to $\Psi$ is equal to the set of all connected components of $\Delta^+$ isomorphic to $\Psi$. It follows that $\Gamma^+\cong\Delta^+$, which finishes our proof.
\end{proof}\\

The following corollary, which is the main result of this section, follows immediately from Theorem \ref{medjusobno odredjivanje opsteg i nenula-stepenog grafa} and Theorem \ref{medjusobno odredjedjivanje pozitivno-stepenog i nenula-stepenog grafa}.

\begin{corollary}\label{sve tri verzije daju isto}
Let $\mathbf G$ and $\mathbf H$ be power-associative loops. Then the following are equivalent:
\begin{enumerate}
\item $\mathbf G$ and $\mathbf H$ have isomorphic power graphs;
\item $\mathbf G$ and $\mathbf H$ have isomorphic $Z^\pm$-power graphs;
\item $\mathbf G$ and $\mathbf H$ have isomorphic $N$-power graphs.
\end{enumerate}
\end{corollary}

\section{The power graph of a torsion-free group}
\label{torziono slobodne grupe}

In this section, for a graph $\Gamma$, with $S_{\Gamma}(x,y)$ we shall denote the set:
$$S_\Gamma(x,y)=\overline N_\Gamma(y)\setminus\overline N_\Gamma(x),$$
where $\overline N_\Gamma(x)$ denotes the closed neighborhood of $x$ in $\Gamma$. Also, for a group $\mathbf G$, with $S_{\mathbf G}(x,y)$ we will denote $S_{\mathcal G^\pm(\mathbf G)}(x,y)$.

\begin{lemma}[{\cite[Lemma 3.1]{cameron3}}]\label{samo su torzione sa jedinicnim}
Let $\mathbf G$ be a group. Then $\mathbf G$ is torsion-free if and only if $\mathcal G^\pm(\mathbf G)$ has an isolated vertex. Moreover, $\mathcal G^\pm(\mathbf G)$ has at most one isolated vertex.
\end{lemma}

\begin{lemma}[{\cite[Lemma 4.1]{cameron3}}]\label{kriterijum za ko koga tuce u torziono slobodnoj}
Let $\mathbf G$ be a torsion-free group, and let $x,y\in G$ be such that $x\not\in\{y,y^{-1}\}$. Then the following holds:
\begin{enumerate}
\item If $x\rightarrowpm y$, then $S_{\mathbf G}(y,x)$ is infinite;
\item If $\mathbf G$ is a group whose every non-identity element is contained in a unique maximal cyclic subgroup, and if $x\simppm y$, then
\begin{align*}
x\rightarrowpm y\text{ if and only if the set }S_{\mathbf G}(x,y)\text{ is finite.}
\end{align*}
\end{enumerate}
\end{lemma}

The complement of a graph $\Gamma=(V,E)$ is the graph $\overline\Gamma=(V,V^{[2]}\setminus E)$. In this section, for an element $x$ of a group $\mathbf G$, with $I_{\mathbf G}(x)$, $O_{\mathbf G}(x)$, and $M_{\mathbf G}(x)$ we shall denote the set of all its direct predecessors without $x^{-1}$,  the set of all its direct successors without $x^{-1}$, and  the set of all its neighbors without $x^{-1}$ in $\vec{\mathcal G}^\pm(\mathbf G)$, respectively, i.e.
\begin{align*}
&I_{\mathbf G}(x)=\{y\in V\setminus\{x^{-1}\}\mid y\rightarrowpm_{\mathbf G} x\},\\
&O_{\mathbf G}(x)=\{y\in V\setminus\{x^{-1}\}\mid x\rightarrowpm_{\mathbf G} y\}\text{ i}\\
&M_{\mathbf G}(x)=I_{\mathbf G}(x)\cup O_{\mathbf G}(x).
\end{align*}
Sometimes we may denote $I_{\mathbf G}(x)$, $O_{\mathbf G}(x)$, and $M_{\mathbf G}(x)$ shortly with $I(x)$, $O(x)$, and $M(x)$, respectively. Further, for a group $\mathbf G$ and its $Z^\pm$-power graph $\Gamma^\pm=\mathcal G^\pm(\mathbf G)$, we introduce the following denotements:
\begin{align*}
&\mathcal I_{\mathbf G}(x)=\Gamma^\pm[I(x)],&&
&\mathcal O_{\mathbf G}(x)=\Gamma^\pm[O(x)],&&
&\mathcal M_{\mathbf G}(x)=\Gamma^\pm[M(x)],\\
&\overline{\mathcal I}_{\mathbf G}(x)=\overline{\Gamma^\pm}[I(x)],&&
&\overline{\mathcal O}_{\mathbf G}(x)=\overline{\Gamma^\pm}[O(x)],&&
&\overline{\mathcal M}_{\mathbf G}(x)=\overline{\Gamma^\pm}[M(x)],
\end{align*}
for the respective induced subgraphs of $\mathcal G^\pm(\mathbf G)$ or its complement. Sometimes we shall write shortly $\mathcal I(x)$, $\mathcal O(x)$, $\mathcal M(x)$, $\overline{\mathcal I}(x)$, $\overline{\mathcal O}(x)$, and $\overline{\mathcal M}(x)$. Note that, for a non-identity element $x$ of a torsion-free group $\mathbf G$, the element $x^{-1}$ is recognizable by the fact that it is the only vertex which has the same closed neighborhood in $\mathcal G^\pm(\mathbf G)$ as the vertex $x$.

\begin{lemma}[{\cite[Lemma 3.3]{cameron3}}]\label{O je komponenta povezanosti od N}
Let $\mathbf G$ be a torsion-free group and let $x$ be a non-identity element $\mathbf G$. Then $\overline{\mathcal O}_{\mathbf G}(x)$ is a connected component of $\overline{\mathcal M}_{\mathbf G}(x)$.
\end{lemma}

\begin{lemma}[{\cite[Lemma 3.4]{cameron3}}]\label{susedstva su izomorfna}
Let $\mathbf G$ be a torsion-free group, and let $\mathbf H$ be a group such that $\mathcal G^\pm(\mathbf G)\cong\mathcal G^\pm(\mathbf H)$. Let $z$ be a non-identity element of $\mathbf G$, and let $\varphi: G\rightarrow H$ be an isomorphism from $\mathcal G^\pm(\mathbf G)$ to $\mathcal G^\pm(\mathbf H)$. Then $\hat\varphi=\varphi\rvert_{N(x)}$ is an isomorphism from $\overline{\mathcal M}_{\mathbf G}(z)$ to $\overline{\mathcal M}_{\mathbf H}\big(\hat\varphi(z)\big)$, and $\overline{\mathcal I}_{\mathbf G}(z)\cong\overline{\mathcal I}_{\mathbf H}\big(\hat\varphi(z)\big)$ and $\overline{\mathcal O}_{\mathbf G}(z)\cong\overline{\mathcal O}_{\mathbf H}\big(\hat\varphi(z)\big)$.
\end{lemma}

By the previous lemma, for an element $z$ of a torsion-free group $\mathbf G$ and for $\varphi\in\Is\big(\mathcal G^\pm(\mathbf G),\mathcal G^\pm(\mathbf H)\big)$, where $\mathbf H$ is any group, holds $\mathcal O_{\mathbf G}(z)\cong\mathcal O_{\mathbf H}\big(\varphi(z)\big)$ and $\mathcal I_{\mathbf G}(z)\cong\mathcal I_{\mathbf H}\big(\varphi(z)\big)$. However, an isomorphism from $\mathcal G^\pm(\mathbf G)$ might not map $O_{\mathbf G}(z)$ onto $O_{\mathbf H}\big(\varphi(z)\big)$, which is going to be confirmed by Lemma \ref{reciprocni prave izomorfizam}.

The {\bf transposed digraph} of a directed graph $\vec\Gamma$ is the digraph ${\vec\Gamma}^T$ such that $x\rightarrow_{\vec\Gamma^T} y$ if $y\rightarrow_{\vec\Gamma}x$. For digraphs $\vec\Gamma=(V_1,E_1)$ and $\vec\Delta=(V_2,E_2)$, bijection $\varphi:V_1\rightarrow V_2$ is an {\bf anti-isomorphism} from $\vec\Gamma$ to $\vec\Delta$ if $\varphi$ is an isomorphism from $\vec\Gamma$ to $\vec\Delta^T$, and if such bijection exists, then we shall say that $\vec\Gamma$ and $\vec\Delta$ are anti-isomorphic.

\begin{lemma}[{\cite[Lemma 6.1]{cameron3}}]\label{reciprocni prave izomorfizam}
Let $a\in\mathbb Q\setminus\{0\}$, and let $\varphi_a:\mathbb Q\rightarrow\mathbb Q$ be defined as: 
$$\varphi(x)=\begin{cases}
0,&x= 0 \vspace{1mm}\\
\displaystyle\frac{a^2}x,&x\neq 0.
\end{cases}$$
Then $\varphi_a$ is an automorphism of $\mathcal G^\pm(\mathbb Q)$ and an isomorphism from $\vec{\mathcal G}(\mathbb Q)$ to $\big(\vec{\mathcal G}(\mathbb Q)\big)^T$. What is more, mapping $\varphi_a$ is an isomorphism from $\mathcal I_{\mathbb Q}(a)$ to $\mathcal O_{\mathbb Q}(a)$.
\end{lemma}

Before we move on with the proof of the next lemma, let us introduce a well known theorem about locally cyclic groups. Its proof can be found in \cite[Glava VIII, Sekcija 30]{kurosh}.

\begin{theorem}\label{sve torziono slobodne lokalno ciklicne su u q}
A torsion-free group is locally cyclic if and only if it is isomorphic to a subgroup of the group of rational numbers.
\end{theorem}

\begin{lemma}\label{q je jedina sa izomorfnim ulaznim i izlaznim}
Let $\mathbf G$ be a torsion-free locally cyclic group such that $\mathcal O_{\mathbf G}(x)\cong\mathcal I_{\mathbf G}(x)$ for all $x\in G$. Then $\mathbf G\cong(\mathbb Q,+)$.  
\end{lemma}

\begin{proof}
Suppose that $\mathbf G\not\cong(\mathbb Q,+)$ and that $\mathcal I_{\mathbf G}(x)\cong\mathcal O_{\mathbf G}(x)$ for all $x\in G$. By Theorem \ref{sve torziono slobodne lokalno ciklicne su u q}, any torsion-free locally cyclic group can be embedded into the group of rational numbers. Notice that $\mathbb Q$ is isomorphic to none of its proper subgroups, because, for every prime $p$, every equation $px=a$ with the unknown $x$ has a solution in $\mathbb Q$. Beside that, every subgroup of the group of rationals is isomorphic to a subgroup of $\mathbb Q$ which contains $1$. Therefore, without loss of generality, one can assume that $\mathbf G<(\mathbb Q,+)$, and that $1\in G$. One can also notice that there is a prime $p$ such that $\frac 1{p^k}\not\in G$ for some $k\in\mathbb N$, and, without loss of generality, we can assume that $\frac 1p\in G$. 

Let $m$ be the maximal natural number such that $\frac 1{p^{m}}\in G$. Let $\preceq$ be the preorder on $O_{\mathbf G}(1)$ such that $x\preceq y$ if $x\rightarrow y$ or $x=y$. Let us introduce the preorder $\preceq$ on $I_{\mathbf G}(1)$ such that $x\preceq y$ if $y\rightarrow x$ or $x=y$. One can easily notice that classes of the preordered set $(O_{\mathbf G}(1),\preceq)$ are two-element sets $\{n,-n\}$, where $n\in\mathbb N\setminus\{1\}$, and that classes of the preordered set $(I_{\mathbf G}(1),\preceq)$ are two-element sets of form $\{\frac 1n,-\frac 1n\}$, for $n\in\mathbb N\setminus\{1\}$. It is not hard to notice that minimal classes of $(O_{\mathbf G}(1),\preceq)$ are the ones containing prime numbers, and that minimal classes of $(I_{\mathbf G}(1),\preceq)$ are the ones containing the multiplicative inverses of prime numbers.

Notice that, for all $x,y\in O_{\mathbf G}(1)$ adjacent in $\mathcal G^\pm(\mathbf G)$, holds $x\rightarrow y$ if and only if $S_{\mathcal O_{\mathbf G}(1)}(x,y)$ is finite set. Also, for $x,y\in I_{\mathbf G}(1)$ holds $y\rightarrow x$ if and only if the set $S_{\mathcal I_{\mathbf G}(1)}(x,y)$ is finite. Let $\varphi: O_{\mathbf G}(1)\rightarrow I_{\mathbf G}(1)$ be an isomorphism from $\mathcal O_{\mathbf G}(1)$ to $\mathcal I_{\mathbf G}(1)$. Based on previously said, $\varphi$ is an isomorphism from $(O_{\mathbf G}(1),\preceq)$ to $(I_{\mathbf G}(1),\preceq)$. Notice that in $(O_{\mathbf G}(1),\preceq)$ each element of any minimal class is contained in some infinite ascending chain $L$ such that every two elements of $L^\downarrow$ are comparable. Namely, if $x$ is an element of a minimal class of $(O_{\mathbf G}(1),\preceq)$, then $x$ or $x^{-1}$ is a prime. Then the sequence $x,x^2,x^3,\dots$ makes up an ascending chain $L$, and, for each $y\in L^\downarrow$, $y$ or $-y$ is a power of the prime $\lvert x\rvert$. Since $(O_{\mathbf G}(1),\preceq)\cong(I_{\mathbf G}(1),\preceq)$, it follows that the same holds for $(I_{\mathbf G}(1),\preceq)$. However, $\frac 1p$ is contained in no such infinite ascending chain of $(I_{\mathbf G}(1),\preceq)$, which is a contradiction with the supposition that $\mathcal O_{\mathbf G}(x)\cong\mathcal I_{\mathbf G}(x)$ for all $x\in G$. This proves the lemma.
\end{proof}

\begin{lemma}\label{pomocna za lokalno ciklicne}
Let $\mathbf G$ be a torsion-free group, and let $x,y,z\in G$ be such that $x,y\in O_{\mathbf G}(z)$, $x\rightarrowpm_{\mathbf G}y$, and $y\not\in\{x,x^{-1}\}$. Let $\mathbf H$ be a group, and let $\varphi: G\rightarrow H$ be an isomorphism from $\mathcal G^\pm(\mathbf G)$ to $\mathcal G^\pm(\mathbf H)$. Then
\begin{center}
$\varphi(x)\rightarrowpm_{\mathbf H}\varphi(y)$ if and only if $\varphi\big(O_{\mathbf G}(z)\big)= O_{\mathbf H}\big(\varphi(z)\big)$.
\end{center}
\end{lemma}

\begin{proof}
Because, by Lemma \ref{O je komponenta povezanosti od N}, $O_{\mathbf G}(z)$ and $O_{\mathbf H}\big(\varphi(z)\big)$ induce connected components of $\overline{\mathcal M}_{\mathbf G}(z)$ and $\overline{\mathcal M}_{\mathbf H}\big(\varphi(z)\big)$, respectively, then, by Lemma \ref{susedstva su izomorfna}, $\varphi\big(O_{\mathbf G}(z)\big)= O_{\mathbf H}\big(\varphi(z)\big)$ if and only if $\varphi(x),\varphi(y)\in O_{\mathbf H}\big(\varphi(z)\big)$. Therefore, it is sufficient to prove that $\varphi(x)\rightarrowpm_{\mathbf H}\varphi(y)$ if and only if $\varphi(x),\varphi(y)\in O_{\mathbf H}\big(\varphi(z)\big)$.

By Lemma \ref{samo su torzione sa jedinicnim}, group $\mathbf H$ is torsion-free too. Since $\{x,x^{-1}\}$, $\{y,y^{-1}\}$, and $\{z,z^{-1}\}$ are pairwise disjoint sets, then elements $x$, $y$, and $z$ have distinct closed neighborhoods. Therefore, $\varphi(x)$, $\varphi(y)$, and $\varphi(z)$ have distinct closed neighborhoods, which implies that sets $\{\varphi(x),\varphi(x^{-1})\}$, $\{\varphi(y),\varphi(y^{-1})\}$, and $\{\varphi(z),\varphi(z^{-1})\}$ are pairwise disjoint.

Suppose that $\varphi(x)\rightarrowpm_{\mathbf H}\varphi(y)$, and suppose that $\varphi(x),\varphi(y)\in I_{\mathbf H}\big(\varphi(z)\big)$. Then follows $\varphi(x),\varphi(y)\rightarrowpm_{\mathbf H}\varphi(z)$. By Lemma \ref{O je komponenta povezanosti od N}, $\overline{\mathcal O}_{\mathbf G}(z)$ is a connected component of $\overline{\mathcal M}_{\mathbf G}(z)$, so $\varphi(O_{\mathbf G}(z))$ induces a connected component of $\overline{\mathcal M}_{\mathbf H}\big(\varphi(z)\big)$ isomorphic to $\overline{\mathcal O}_{\mathbf G}(z)$. By Lemma \ref{kriterijum za ko koga tuce u torziono slobodnoj}, $S_{\mathcal O_{\mathbf G}(z)}(y,x)$ is infinite, which implies that $S_{\varphi(\mathcal O_{\mathbf G}(z))}\big(\varphi(y),\varphi(x)\big)$ is infinite. However, it is not hard to see that $\varphi(x)\rightarrowpm_{\mathbf H}\varphi(y)$ and $\varphi(y)\rightarrowpm_{\mathbf H}\varphi(z)$ implies that $S_{\varphi(\mathcal O_{\mathbf G}(z))}\big(\varphi(y),\varphi(x)\big)$ is finite, which is a contradiction. Thus, $\varphi(x)\rightarrowpm_{\mathbf H}\varphi(y)$ implies $\varphi(x),\varphi(y)\in O_{\mathbf H}\big(\varphi(z)\big)$, and it remains to prove the other implication as well.

Suppose that $\varphi(x),\varphi(y)\in O_{\mathbf H}\big(\varphi(z)\big)$, which implies $\varphi\big(O_{\mathbf G}(z)\big)=O_{\mathbf H}\big(\varphi(z)\big)$. Suppose now that $\varphi(x)\not\rightarrowpm_{\mathbf H}\varphi(y)$. That implies $\varphi(y)\rightarrowpm_{\mathbf H}\varphi(x)$, i.e. $\varphi(z)\rightarrowpm_{\mathbf H}\varphi(y)$ i $\varphi(y)\rightarrowpm_{\mathbf H}\varphi(x)$. By Lemma \ref{kriterijum za ko koga tuce u torziono slobodnoj}, it follows that $S_{\mathcal O_{\mathbf H}(\varphi(z))}\big(\varphi(y),\varphi(x)\big)$ is finite. On the other hand, the fact that $S_{\mathcal O_{\mathbf G}(z)}(y,x)$ is infinite implies that $S_{\varphi(\mathcal O_{\mathbf G}(z)}\big(\varphi(y),\varphi(x)\big)=S_{\mathcal O_{\mathbf H}(\varphi(z))}\big(\varphi(y),\varphi(x)\big)$ is infinite, which is a contradiction. Thus the Lemma has been proved.
\end{proof}

\begin{theorem}[{\cite[Teorema 1.5]{cameron3}}]\label{glavna za racionalne na n}
Let $n\in\mathbb N$, and let $\mathbf G$ be a group such that $\mathcal G^\pm(\mathbf G)\cong\mathcal G^\pm(\mathbb Q)$. Then $\vec{\mathcal G}^\pm(\mathbf G)\cong\vec{\mathcal G}^\pm(\mathbb Q^n)$. Moreover, every isomorphism from $\mathcal G^\pm(\mathbf G)$ to $\mathcal G^\pm(\mathbb Q)$ is an isomorphism or an anti-isomorphism from $\vec{\mathcal G}^\pm(\mathbf G)$ to $\vec{\mathcal G}^\pm(\mathbb Q)$.
\end{theorem}

\begin{corollary}
Let $\mathbf G$ be a group such that $\mathcal G^\pm(\mathbf G)\cong\mathcal G^\pm(\mathbb Q)$. Then $\mathbf G\cong(\mathbb Q,+)$.
\end{corollary}

\begin{proof}
Because $\mathcal G^\pm(\mathbf G)\cong\mathcal G^\pm(\mathbb Q)$, then, by Theorem \ref{glavna za racionalne na n}, $\vec{\mathcal G}^\pm(\mathbf G)\cong\vec{\mathcal G}^\pm(\mathbb Q)$. It follows that $\mathbf G$ is a locally cyclic group such that, for all $x\in G$, holds $\mathcal O_{\mathbf G}(x)\cong\mathcal I_{\mathbf G}(x)$. Therefore, by Lemma \ref{q je jedina sa izomorfnim ulaznim i izlaznim}, follows that $\mathbf G$ is isomorphic to the group of rational numbers.
\end{proof}

\begin{lemma}[{\cite[Lemma 7.1]{cameron3}}]\label{komponenta povezanosti je lokalno ciklicna grupa}
Let $\mathbf G$ be a torsion-free group of nilpotentcy class $2$, and let $C\subseteq G$ induce a connected component of $\mathcal G^\pm(\mathbf G)$. Then $C\cup\{e\}$ is the universe of a locally cyclic subgroup of $\mathbf G$.
\end{lemma}

\begin{theorem}
Let $\mathbf G$ be a torsion-free group of nilpotency class $2$, and let $\mathbf G$ be a group such that $\mathcal G^\pm(\mathbf G)\cong \mathcal G^\pm(\mathbf H)$. Then $\vec{\mathcal G}^\pm(\mathbf G)\cong\vec{\mathcal G}^\pm(\mathbf H)$.
\end{theorem}

\begin{proof}
Let $\varphi: G\rightarrow H$ be an isomorphism from $\mathcal G^\pm(\mathbf G)$ to $\mathcal G^\pm(\mathbf H)$. Then $\varphi$ maps every connected component of $\mathcal G^\pm(\mathbf G)$ onto some connected component of $\mathcal G^\pm(\mathbf H)$. Let $C$ non-trivial connected component of $\mathcal G^\pm(\mathbf G)$, and let us denote $\varphi(C)$ with $D$. By Lemma \ref{komponenta povezanosti je lokalno ciklicna grupa}, $C\cup\{e_{\mathbf G}\}$ is the universe of a locally cyclic subgroup of $\mathbf G$, that we shall denote with $\hat{\mathbf C}$. Let $x,y\in C$ such that $x\rightarrowpm_{\mathbf G} y$ and $y\not\in\{x,x^{-1}\}$.

Let us show that $\varphi|_C$ is an isomorphism or an anti-isomorphism from graph $\big(\vec{\mathcal G}^\pm(\mathbf G)\big)[C]$ to $\big(\vec{\mathcal G}^\pm(\mathbf H)\big)[D]$. Let $u,v\in C$ be such that $u\rightarrowpm_{\mathbf G}v$. Since $\varphi$ is an isomorphism from $\Gamma$ to $\Delta$, then $\varphi(u)\rightarrowpm_{\mathbf H}\varphi(v)$ or $\varphi(v)\rightarrowpm_{\mathbf H}\varphi(u)$.  Obviously, if $v\in\{u,u^{-1}\}$, then so does  $\varphi(v)\in\{\varphi(u),\varphi(u^{-1})\}$ because in this case $u\equiv_{\mathcal G^\pm(\mathbf G)} v$ and $\varphi(u)\equiv_{\mathcal G^\pm(\mathbf H)}\varphi(v)$, so suppose that $v\not\in\{u,u^{-1}\}$.  Because $\hat{\mathbf C}$ is a locally cyclic subgroup of $\mathbf G$, there exists $w\in C$ such that $x,u\in\langle w\rangle$. Suppose further that $w\not\in\{x,x^{-1},u,u^{-1}\}$. In this case, by Lemma \ref{pomocna za lokalno ciklicne}, $\varphi(x)\rightarrowpm_{\mathbf H}\varphi(y)$ if and only if $\varphi\big(O_{\mathbf G}(w)\big)= O_{\mathbf H}\big(\varphi(w)\big)$, which is, by the same lemma, equivalent to $\varphi(u)\rightarrowpm_{\mathbf H}\varphi(v)$. It is shown in a similar manner that the same equivalence holds when $w\not\in\{x,x^{-1},u,u^{-1}\}$, so $\varphi|_C$ is an isomorphism or an anti-isomorphism from $\big(\vec{\mathcal G}^\pm(\mathbf G)\big)[C]$ to $\big(\vec{\mathcal G}^\pm(\mathbf H)\big)[D]$.

Next we prove that $\big(\vec{\mathcal G}^\pm(\mathbf G)\big)[C]\cong\big(\vec{\mathcal G}^\pm(\mathbf H)\big)[D]$, which obviously holds if $\varphi|_C$ is an isomorphism. If $\varphi|_C$ is an anti-isomorphism, then, for each $x\in C$, $O_{\mathbf G}(x)$ is being mapped onto $I_{\mathbf H}\big(\varphi(x)\big)$, and $I_{\mathbf G}(x)$ is being mapped into $O_{\mathbf H}\big(\varphi(x)\big)$. It follows that $\mathcal O_{\mathbf G}(x)\cong\mathcal I_{\mathbf G}(x)$, which, by Lemma \ref{q je jedina sa izomorfnim ulaznim i izlaznim}, implies $\hat{\mathbf C}\cong(\mathbb Q,+)$. Then, by Lemma \ref{reciprocni prave izomorfizam}, follows $\big(\vec{\mathcal G}^\pm(\mathbf G)\big)[C]\cong\big(\vec{\mathcal G}^\pm(\mathbf H)\big)[D]$.

Now we have proved that, for every connected component of $C$ to $\mathcal G^\pm(\mathbf G)$, $\big(\vec{\mathcal G}^\pm(\mathbf G)\big)[C]\cong\big(\vec{\mathcal G}^\pm(\mathbf H)\big)\big[\varphi(C)\big]$. By that, groups $\mathbf G$ and $\mathbf H$ have isomorphic directed $Z^\pm$-power graphs, which proves the theorem.
\end{proof}


\end{document}